\documentclass[11pt,oneside,reqno]{article}

\usepackage{amsmath,amsthm,amssymb,color}
\usepackage{array}
\usepackage{bbm}
\usepackage{mathrsfs}
\usepackage[breaklinks=true]{hyperref}

\numberwithin{equation}{section}
\allowdisplaybreaks[4]

\theoremstyle{plain}
\newtheorem{theorem}{Theorem}[section]
\newtheorem{proposition}[theorem]{Proposition}
\newtheorem{lemma}[theorem]{Lemma}
\newtheorem{corollary}[theorem]{Corollary}

\theoremstyle{definition}
\newtheorem{definition}[theorem]{Definition}

\newtheorem{remark}[theorem]{Remark}

\makeatletter

\def\^#1{\ifmmode {\mathaccent"705E #1} \else {\accent94 #1} \fi}
\def\~#1{\ifmmode {\mathaccent"707E #1} \else {\accent"7E #1} \fi}

\def\*#1{#1^\ast}
\edef\-#1{\noexpand\ifmmode {\noexpand\bar{#1}} \noexpand\else \-#1\noexpand\fi}
\def\>#1{\vec{#1}}
\def\.#1{\dot{#1}}

\def\atop{\@@atop}
\def\%#1{\mathcal{#1}}

\renewcommand{\leq}{\leqslant}
\renewcommand{\geq}{\geqslant}
\renewcommand{\phi}{\Varphi}
\newcommand{\eps}{\varepsilon}
\newcommand{\D}{\Delta}
\newcommand{\eq}{\eqref}

\newcommand{\dtv}{\mathop{d_{\mathrm{TV}}}}

\newcommand{\bigo}{\mathrm{O}}

\def\tsfrac#1#2{{\textstyle\frac{#1}{#2}}}

\newcommand{\I}{\mathbb{I}}

\newcommand{\IE}{\mathbbm{E}}
\newcommand{\IP}{\mathbbm{P}}

\newcommand{\law}{\mathscr{L}}

\newcommand{\IN}{\mathbbm{N}}
\newcommand{\IZ}{\mathbbm{Z}}
\newcommand{\IR}{\mathbbm{R}}

\def\be#1{\begin{equation*}#1\end{equation*}}
\def\ben#1{\begin{equation}#1\end{equation}}

\def\ba#1{\begin{align*}#1\end{align*}}
\def\ban#1{\begin{align}#1\end{align}}

\def\bklr#1{\bigl(#1\bigr)}

\def\bbbkle#1{\biggl[#1\biggr]}

\def\abs#1{\vert#1\vert}

\newcount\minute
\newcount\hour
\newcount\hourMins
\def\now{%
\minute=\time%
\hour=\time \divide \hour by 60%
\hourMins=\hour \multiply\hourMins by 60%
\advance\minute by -\hourMins%
\zeroPadTwo{\the\hour}:\zeroPadTwo{\the\minute}%
}
\def\zeroPadTwo#1{\ifnum #1<10 0\fi#1}

\renewcommand\section{\@startsection {section}{1}{\z@}%
{-3.5ex \@plus -1ex \@minus -.2ex}%
{1.3ex \@plus.2ex}%
{\center\small\sc\mathversion{bold}\MakeUppercase}}

\def\subsection#1{\@startsection {subsection}{2}{0pt}%
{-3.5ex \@plus -1ex \@minus -.2ex}%
{1ex \@plus.2ex}%
{\bf\mathversion{bold}}{#1}}

\def\subsubsection#1{\@startsection{subsubsection}{3}{0pt}%
{\medskipamount}%
{-10pt}%
{\normalsize\itshape}{\kern-2.2ex. #1.}}

\def\blfootnote{\xdef\@thefnmark{}\@footnotetext}

\makeatother

\def\t#1{^{(#1)}}
\newcommand\ed{\stackrel{d}{=}}

\newcommand\G{\Gamma}
\newcommand\NB{\textrm{NB}}
\newcommand\K{\textrm{K}}
\newcommand\dr{D^{(r)}}

\begin{document}

\title{\sc\bf\large\MakeUppercase{Power laws in preferential attachment graphs
and Stein's method for the negative binomial distribution.}}
\author{\sc %Erol Pek\"oz, Adrian R\"ollin and 
Nathan Ross}
\date{\it %Boston University, National University of Singapore, 
University of California at
Berkeley}
\maketitle

\begin{abstract} 
For a family of linear preferential attachment graphs, 
we provide rates of convergence for the total variation distance 
between the degree of a randomly chosen
vertex and an appropriate power law distribution as the number of vertices tends to infinity. 
Our proof uses a new formulation of Stein's method for the
negative binomial distribution, which stems from
a distributional transformation that has 
the negative binomial distributions
as the only fixed points.
\end{abstract}


\section{Introduction}
Preferential attachment random graphs were introduced in~\cite{baal99} 
as a stochastic mechanism to explain power law degree distributions
empirically observed in real world networks such as the world wide web.
These graphs evolve by sequentially adding vertices and edges in a
random way so that connections to
vertices with high degree are favored. 
There has been much interest in properties of these models and their
many embellishments; the draft text~\cite{vdh12} is probably the
best survey of this vast literature.
Like the seminal work \cite{baal99} (and the mathematically precise formulation~\cite{brst01}),
much of this research
is devoted to showing that if the number of vertices of the graph is large,
then the proportion of vertices having degree $k$
approximately decays as $c_\gamma k^{-\gamma}$ for some constant $c_\gamma$ and $\gamma>1$;
the so-called power law behavior.

Our main result in this vein is Theorem~\ref{thmmain} below, 
which, for a family of linear preferential attachment graphs, 
provides rates of convergence for the total variation distance 
between the degree of a randomly chosen
vertex and an appropriate power law distribution as the number of vertices tends to infinity.
The result is new and
the method of proof is also of interest since it differs substantially
from proofs of similar results (e.g.\ Section~8.5 of~\cite{vdh12}).
Our proof of Theorem~\ref{thmmain} uses a new formulation of Stein's method
for the
negative binomial distribution, Theorem~\ref{thmnbst} below
(see~\cite{ros11} and
references therein for a basic introduction to Stein's method).
The result stems from
a distributional transformation that has negative binomial distributions
as the only fixed points (we shall shortly see the
relationship between the negative binomial distribution
and power laws). Similar strategies have recently found
success in analyzing degree distributions in preferential attachment models, see~\cite{prr11}
and Section~6 of~\cite{prr10}; the latter is a special case of our
results and is the template for our proofs.
The remainder of the introduction is devoted to stating our results in greater detail. 

First we define the family of preferential attachment models we study;
these are the same models studied in Chapter~8 of~\cite{vdh12}, which
are a generalization of the models first defined in~\cite{brst01}, which in turn
are a formalization of the heuristic models described in~\cite{baal99}.
The family of models is parameterized by $m\in\IN$ and $\delta>-m$.
For $m=1$ and given $\delta$, the model 
starts with one vertex with a single loop where
one end of the loop contributes to the ``in-degree" and the other to the
``out-degree." Now, for $2\leq k\leq n$, given the graph with $k-1$ vertices, add
vertex $k$ along with an edge emanating ``out" from $k$ ``in" to a random vertex chosen from 
the set $\{1,\ldots,k\}$ with probability
proportional to the total degree of that vertex plus $\delta$, where initially vertex $k$ 
has degree one. That is, at step $k$, the chance
that vertex $k$ connects to itself is $(\delta+1)/(k(2+\delta)-1))$.  
%We think of the out-degree of a vertex as having weight $1+\delta$,
%and the in-degrees as having weight one.
After $n$ steps of this process, we denote the resulting random graph
by $G_n^{1,\delta}$.

For $m>1$, we define $G_n^{m,\delta}$ by first generating $G_{nm}^{1,\delta/m}$, and then
``collapsing" consecutive vertices into groups of size $m$, starting from the first vertex,
and retaining all edges. 
Note that with this setup, it is possible for a vertex to connect to itself or other
vertices more than once and as many as $m$ times (in fact the first vertex always 
consists of $m$ loops) and all of these connections contribute
to the in- and out-degree of a vertex (e.g.\ the first vertex has both in- and out-degree $m$).

Here and below, we think of $\delta$ and $m$ as fixed
and let $W_n$ be the in-degree of a randomly chosen vertex from $G_n^{m,\delta}$.
We provide a bound on the total variation distance between
$W_n$ and a limiting distribution
which is a mixture of negative binomial distributions. 
For $r>0$ and $0< p\leq 1$, we say $X\sim\NB(r,p)$
if
\ba{
\IP(X=k)=\frac{\Gamma(r+k)}{k!\Gamma(r)}(1-p)^kp^r, \hspace{5mm} k=0,1,\ldots
}

\begin{definition}\label{defk}
For $m\in \IN$, $\delta>-m$ and $U$ uniform on $(0,1)$, denote the mixture distribution
 $\NB(m+\delta,U^{1/(2+\delta/m)})$ by $\K(m, \delta)$.
\end{definition}
 
For our main result, we define the total variation distance between two non-negative integer valued random variables $X,Y$
\ban{
\dtv(\law(X),\law(Y))&=\sup_{A\subseteq \IZ_+}|\IP(X\in A)-\IP(Y\in A)| \label{dtv} \\
	&=\tsfrac{1}{2}\sum_{k\in\IZ_+} |\IP(X=k)-\IP(Y=k)|. \label{dtv1}
}
where here and below $\IZ_+=\{0,1,\ldots\}$.
\begin{theorem}\label{thmmain}
If $W_n$ is the in-degree of a randomly chosen vertex from the preferential attachment
graph $G_n^{m,\delta}$
and $\K(m,\delta)$ is the mixed negative binomial distribution of Definition~\ref{defk}, then
for some constant $C_{m,\delta}$,
\ba{
\dtv(\law(W_n),\K(m,\delta))\leq C_{m,\delta}\frac{\log(n)}{n}.
}
\end{theorem}

To see the power law behavior of $\K(m,\delta)$,
we record the following easy result 
which is a more standard representation of
$\K(m,\delta)$ through its point probabilities.
The proof follows from direct computation and 
then Stirling's formula (or Lemma~\ref{lemgamrat} below). 
These formulas with additional discussion 
are also found in
Section~8.3 of~\cite{vdh12},
specifically~(8.3.2), and~(8.3.9-10). 
The representation of $\K(m,\delta)$ as a mixture of negative
binomial distributions does not seem to be well known.

\begin{lemma}
If $m\in\IN$, $\delta>-m$, and $Z\sim\K(m,\delta)$, then
for $l=0,1,\ldots$,
\ba{
\IP(Z=l)&=\left(2+\tsfrac{\delta}{m}\right)\frac{\G(l+m+\delta)\G\left(m+2+\delta+\tsfrac{\delta}{m}\right)}{\G(m+\delta)\G\left(l+m+3+\delta+\tsfrac{\delta}{m}\right)},}
and for $c_{m,\delta}=(2+\delta/m)\Gamma\left(m+2+\delta+\tsfrac{\delta}{m}\right)/\G(m+\delta)$,
\ba{
\IP(Z=k)&\asymp \frac{c_{m,\delta}}{k^{3+\delta/m}}%\left(1+\bigo(1/k)\right)
\,\,\,\text{as $k\to\infty$}.
}
\end{lemma}
Before discussing our Stein's method result, we make a few final remarks.
The usual mathematical statement
implying power law behavior of the degrees of a random graph 
in this setting is that
the empirical degree distribution converges to $\K(m,\delta)$ in probability 
(Theorem~8.2 of~\cite{vdh12}). Such a result 
implies the total variation distance in Theorem~\ref{thmmain} tends to zero (see Exercise~8.14 of~\cite{vdh12}),
but does not provide a rate.
Another result similar to Theorem~\ref{thmmain} is Proposition~8.4 of~\cite{vdh12} which states
that for $Z\sim\K(m,\delta)$
\ba{
\abs{\IP(W_n=k)-\IP(Z=k)}\leq C/n,
}
which according to~\eq{dtv1} neither implies nor is implied by Theorem~\ref{thmmain}.
Finally, regarding other preferential attachment models,
our results can likely be extended to some other models where
the limiting distribution is $\K(m,\delta)$, for example where the update
rule is that we consider here, but the starting graph is not. 
For other preferential attachment graphs
where the limiting degree distribution is
not $\K(m,\delta)$ (such as those of~\cite{rtv07}),
it may be possible to prove analogs of Theorem~\ref{thmmain} 
using methods similar to ours, but we do not pursue this here.

To state our general result  which we use to prove Theorem~\ref{thmmain}, 
we first define a distributional transformation.
For $r>0$ and $n\geq1$ let $U_{r,n}$ be a random variable having the distribution of
the number of white balls drawn in $n-1$ draws in 
a standard P\'olya urn scheme starting with $r$ ``white balls" and $1$ black ball.
That is, for fixed $r$, we construct $U_{r,n}$ sequentially by setting $U_{r,1}=0$, and for $k\geq1$,
\ban{ 
\IP(U_{r,k+1}=U_{r,k}+1|U_{r,k})=1-\IP(U_{r,k+1}=U_{r,k}|U_{r,k})=\frac{r+U_{r,k}}{r+k}. \label{urn}
}
Also, for a non-negative integer valued random variable $X$ with finite mean,
we say $X^s$ has the \emph{size bias} distribution of $X$ if
\ba{
\IP(X^s=k)=\frac{k\IP(X=k)}{\IE X}, \hspace{5mm} k=1,2,\ldots
}
\begin{definition}\label{req}
Let $X$ be a non-negative integer valued random variable with finite mean and let $X^s$ denote
a random variable having the size bias distribution of $X$. We say 
the random variable $X^{*_r}$ has the \emph{$r$-equilibrium} transformation if
\ba{
X^{*_r}\ed U_{r,X^s},
}
where we understand $U_{r,X^s}$ to mean $\law(U_{r,X^s}|X^s=k)=\law(U_{r,k})$.  
\end{definition}
As we shall see below in Corollary~\ref{cornb},
$X^{*_r}\ed X$ if and only if $X\sim\NB(r,p)$ for some $0<p< 1$.  
Thus if some non-negative integer valued random variable $W$ 
has approximately the same distribution as $W^{*_r}$, it
is plausible that $W$ is approximately distributed
as a negative binomial distribution. The next result
makes this heuristic precise. Here and below we denote the indicator
of an event $B$ by $\I_B$ or $\I[B]$.

\begin{theorem}\label{thmnbst}
Let $W$ be a non-negative integer valued random variable
with $\IE W=\mu$. Let also $r>0$ and $W^{*_r}$ be coupled to $W$ and
have the $r$-equilibrium transformation of Definition~\ref{req}.
If $p=r/(r+\mu)$ and $c_{r,p}=\min\{(r+2)(1-p),2-p\}\leq 2$,
then for an event $B$
\ba{
\dtv(\law(W),\NB(r,p))&\leq c_{r,p} \IE\left[\I_B\abs{W^{*_r}-W}\right]+2(e\max\{1,r\}+1)\IP(B^c) \\
	&\leq 2(e\max\{1,r\}+1)\IP(W^{*_r}\not=W).
}
\end{theorem}

\begin{remark}\label{remod}
Analogs of Theorem~\ref{thmnbst} for other distributions which use fixed points of 
distributional transformations 
are now well established in the Stein's method literature.
For example, the book~\cite{bhj92} develops Stein's method for Poisson approximation
using the fact that a non-negative integer valued random variable $X$ with
finite mean has the Poisson distribution if and only if
$X\ed X^s-1$. Also there is the zero bias transformation
for the normal distribution~\cite{gore97}, the equilibrium transformation for the exponential distribution~\cite{rope10},
a less standard distribution~\cite{prr11},
and the special case where $r=1$ above, the discrete equilibrium transformation for the geometric distribution~\cite{prr10}
(see also~\cite{pek96} for an unrelated transformation used for geometric approximation).
\end{remark}
\begin{remark}
The fact that negative binomial distributions are the fixed points of
the $r$-equilibrium transformation is the discrete analog  
of the fact, perhaps more familiar,
that a non-negative random variable $X$ has the gamma distribution with shape parameter
$\alpha$ if and only if 
\ba{
X\ed B_{\alpha,1}X^s,
}
where $B_{\alpha,1}$ is a beta variable with density $\alpha x^{\alpha-1}$ for $0<x<1$ independent of $X^s$; see~\cite{piro12}.
\end{remark}

The layout of the remainder of the article is as follows. 
In Section~\ref{secnb} we develop Stein's method for the
negative binomial distribution using the $r$-equilibrium transformation
and prove Theorem~\ref{thmnbst}. In Section~\ref{secpa}
we use Theorem~\ref{thmnbst} to prove Theorem~\ref{thmmain}.

\section{Negative binomial approximation}\label{secnb}
The proof of Theorem~\ref{thmnbst} roughly follows
the usual development of Stein's method of distributional approximation
using fixed points of distributional transformations (see the references of Remark~\ref{remod}).
Specifically, if $W$ is a non-negative integer valued random variable of interest and $Y$ has the negative binomial distribution, 
then using the definition~\eq{dtv} we want to bound
$|\IP(W\in A)-\IP(Y\in A)|$
uniformly for $A\subseteq \IZ_+$.
Typically, this program has three components.
\begin{enumerate}
\item[1.] Define a \emph{characterizing operator} $\mathcal{A}$ for the negative binomial distribution which
has the property that
\ba{
\IE \mathcal{A}g(Y)=0
}
for all $g$ in a large enough class of functions if and only if $\law(Y)\sim \NB(r,p)$.

\item[2.] For $A\subseteq\IZ_+$, define $g_A$ to solve
\ban{\label{cheq}
\mathcal{A}g_A(k)=\I[k\in A]-\IP(Y\in A).
}
\item[3.] Using \eq{cheq}, note that
\ba{
\abs{\IP(W\in A)-\IP(Y\in A)} = \abs{\IE \mathcal{A}g_A(W)}.
}
Now use properties of the solutions $g_A$ and the distributional
transformation to bound the right side of this equation.
\end{enumerate}

Obviously there must be some relationship between 
the characterizing operator of Item~1 and the distributional transformation of Item~3; this is typically the subtle 
part of the program above. For Item~1, we use the characterizing operator for the negative binomial distribution as defined in~\cite{brph99}.
\begin{theorem}\label{thmbp}\cite{brph99}
If $W\geq0$ has a finite mean, then
$W\sim\NB(r,p)$ if and only if
\ban{
\IE[(1-p)(r+W) g(W+1)-Wg(W)]=0 \label{nbchar}
} 
for all bounded functions $g$.
\end{theorem}

We need to develop the connection between the characterizing operator of Theorem~\ref{thmbp} and
the $r$-equilibrium transformation. To this end, for a function $g$
define 
\ba{
\dr g(k)=\left(k/r+1\right)g(k+1)-(k/r)g(k),
}
and note that the negative binomial characterizing operator of~\eq{nbchar} can be written
\ban{
r(1-p)\dr g(W)-pWg(W). \label{chop1}
}
The key relationship is the following. 
\begin{lemma}\label{lemnbrt}
If the integer valued random variable $X\geq0$
has finite mean $\mu>0$, $X^{*_r}$ has the $r$-equilibrium distribution of $X$,
and $g$ is a function such that the expectations
below are well defined, then 
\ba{
\mu \IE \dr g(X^{*_r}) = \IE X g(X).
}
\end{lemma}
\begin{proof}
We show that 
\ban{
\IE \dr g(U_{r,n})=g(n), \label{gdr1}
}
which, using the definition of the size bias distribution implies that
\ba{
\mu\IE \dr g(X^{*_r})=\mu \IE g(X^s)=\IE Xg(X),
}
as desired. To show~\eq{gdr1}, we use induction on $n$. The equality is obvious for $n=1$ 
since $U_{r,1}=0$.  Assume that \eq{gdr1} holds for $n$ and we  show it holds for $n+1$.
By conditioning on the previous step in the urn process defining $U_{r,n+1}$ and using~\eq{urn},
we find for a function $f$ such that the expectations below
are well defined,
\ba{
\IE f(U_{r,n+1})= \frac{1}{r+n}\IE (U_{r,n}+r)f(U_{r,n}+1)+\IE \left(1-\frac{U_{r,n}+r}{r+n}\right)f(U_{r,n}).
}
Combining this equality with the induction hypothesis in the form
\ba{
\IE (U_{r,n}+r)f(U_{r,n}+1)=rf(n)+\IE U_{r,n}f(U_{r,n}),
}
yields
\ba{
\IE f(U_{r,n+1})=\frac{r}{r+n} f(n)+\frac{n}{r+n} \IE f(U_{r,n}).
}
Now taking $f=\dr g$ and using the induction hypothesis again yields~\eq{gdr1}.
\end{proof}

We now record
the following result which, while not necessary for
the proof of Theorem~\ref{thmnbst}, underlies our whole approach for
negative binomial approximation.

\begin{corollary}\label{cornb}
If the integer valued random variable $X\geq0$ is such that $\IE X=r(1-p)/p$ for some
$0<p<1$, then
$X\sim\NB(r,p)$ if and only if
\ba{
X\ed X^{*_r}.
}
\end{corollary}
\begin{proof}
If $X\ed X^{*_r}$ then
combining Theorem~\ref{thmbp} and Lemma~\ref{lemnbrt}, noting
the representation~\eq{chop1}, we easily see that $X\sim\NB(r,p)$.

Conversely, assume $Y\sim\NB(r,p)$, and we show $Y^{*_r} \ed Y$
using the method of moments.
According to (4.3) on Page 178 of~\cite{joko77},
\ba{
\IE \left[U_{r,n}(U_{r,n}-1)\cdots(U_{r,n}-k+1)\right]=\frac{r (n-1)\cdots(n-k)}{r+k},
}
which implies that for $X$ with finite $k+1$ moments and $\IE X=r(1-p)/p$,
\ba{
\IE \left[X^{*_r}\cdots(X^{*_r}-k+1)\right]&=\frac{r \IE \left[(X^s-1)\cdots(X^s-k)\right]}{(r+k)}, \\	
	&=\frac{p}{1-p}\frac{\IE \left[X(X-1)\cdots(X-k)\right]}{(r+k)}.
}
Now from display (2.29) on Page 84 of~\cite{joko77},
if $Y\sim\NB(r,p)$,
then
\ba{
\IE[Y\cdots(Y-k+1)]=r\cdots(r+k-1)\left(\frac{1-p}{p}\right)^k.
}
Combining this with the calculation above,
we find that for all $k\geq1$,
\ba{
\IE [Y^{*_r}\cdots (Y^{*_r}-k+1)]=\IE [Y\cdots(Y-k+1)].
}
Since $Y$ has a well behaved moment generating function (i.e.\ exists in a neighborhood around zero), the moment sequence determines
the distribution and so $Y\ed Y^{*_r}$, as desired.
\end{proof}

%\begin{remark}
%The formulation of the theorem is a bit strange, since the mean of $X$ won't determine the value of $r$. 
%However, if $X\ed X^{*_r}$, then $p$ is determined by the mean of $X$ (and of course $r$ and $p$ are known for
%the converse).
%\end{remark}

The next two lemmas take care of Item~2 in the program outlined above, and obtain
the properties of the solution for Item~3. We prove
Theorem~\ref{thmnbst} immediately after the lemmas.
For a function $g:\IZ_+\to\IR$, define $\D g(k)=g(k+1)-g(k)$.

\begin{lemma}\label{lemnbstbd}
If $Y\sim\NB(r,p)$ and for $A\subseteq \IZ_+$, $g:=g_A$ satisfies the Stein equation
\ban{
(1-p)(r+k) g(k+1)-kg(k)=\I[k\in A]-\IP(Y\in A), \label{nbstsol}%\hspace{5mm} k=0,1\ldots \label{nbstsol}
} 
then for $k=0,1\ldots$,
\ba{
&\abs{(k+1)g(k+1)}\leq\frac{\max\{1,r\}e}{p} \text{\, and \, } \abs{\Delta g(k)}\leq \min\left\{\tsfrac{1}{(1-p)(r+k)}, \tsfrac{1}{k}\right\}.
}
\end{lemma}
\begin{proof}
The second assertion bounding $\abs{\Delta g(k)}$ is Theorem~2.10 applied to Example~2.9 of \cite{brxi01}.
For the first assertion, note that
\ba{
&(k+1)g(k+1)=\frac{\left[\IP(Y\in A, Y\leq k)-\IP(Y\in A)\IP(Y\leq k)\right]}{\IP(Y=k+1)} \\
&=\frac{\left[\IP(Y\in A, Y\leq k)\IP(Y\geq k+1)-\IP(Y\in A, Y\geq k+1)\IP(Y\leq k)\right]}{\IP(Y=k+1)},
}
so we find
\ban{
\abs{(k+1)g(k+1)}\leq \frac{\IP(Y\geq k+1)\IP(Y\leq k)}{\IP(Y=k+1)}, \label{gbd1}%\leq  \frac{\min\{\IP(Y\geq k+1),\IP(Y\leq k)\}}{\IP(Y=k+1)}
}
and the bound also holds with either term alone in the numerator.

If $r=1$ (the geometric distribution), then we can compute~\eq{gbd1} exactly 
as $(1-(1-p)^{k+1})/p\leq 1/p$, as desired.
If $0<r<1$, then Proposition~1(b) of~\cite{kla00}
implies that 
$\IP(Y\geq k+1)/\IP(Y=k+1)\leq 1/p$, which implies the result in this case.

If $r>1$, then we bound~\eq{gbd1} in three cases:
$k+1\geq r(1-p)/p$, $k+1\leq (r-1)(1-p)/p$,
and $(r-1)(1-p)/p+1\leq k+1\leq r(1-p)/p-1$.
For the first case, Proposition~1(b) of~\cite{kla00} implies
that for $k+1\geq r(1-p)/p$,
\ban{
\frac{\IP(Y\geq k+1)}{\IP(Y=k+1)}\leq\left(1-(1-p)\frac{k+1+r}{k+2}\right)^{-1}. \label{gtth}
}
The right hand side is decreasing in $k$, so setting 
$k+1=r(1-p)/p$ and simplifying,
we find that for $k+1\geq r(1-p)/p$,~\eq{gtth} is bounded by
$r/p-r+1\leq r/p$, as desired.  For the other two cases, we use the
representation (see e.g.~(2.27) of~\cite{adjo06})
\ba{
\IP(Y\leq k)=\frac{\Gamma(r+k+1)}{\Gamma(r)\Gamma(k+1)}\int_0^p u^{r-1} (1-u)^k du,
}
which yields that~\eq{gbd1} is bounded by
\ban{
\frac{(k+1)\int_0^p u^{r-1} (1-u)^k du}{p^r(1-p)^{k+1}}. \label{smv1}
}
The maximum of the integrand is achieved at $p_*=(r-1)/(r+k-1)$ and if
$k+1\leq (r-1)(1-p)/p$, then $p_*\geq p$ which implies that
\ba{
\int_0^p u^{r-1} (1-u)^k du\leq p^r (1-p)^k,
}
and thus that~\eq{smv1} is bounded by $(k+1)/(1-p)\leq (r-1)/p\leq r/p$ due
to the restriction on the value of $k$. 

Finally, assume $(r-1)(1-p)/p+1\leq k+1\leq r(1-p)/p-1$ and note
that in order for such $k$ to exist, $0<p\leq 1/3$ and we assume this for the 
remainder of the proof. With $p_*$ as above,
\ba{
\int_0^p u^{r-1} (1-u)^k du\leq p p_*^{r-1} (1-p_*)^k,
}
and the lower bound on the range of $k$ implies that 
$p_*\leq p$ and so we find~\eq{smv1}
is bounded above by
\ban{
\frac{(k+1)(1-p_*)^k}{(1-p)^{k+1}}\leq \frac{r}{p}\left(\frac{1-p_*}{1-p}\right)^k. \label{ugh}
}
Recalling that $1-p \leq 1- p_*=k/(r+k-1)$, it is easy to see that~\eq{ugh} is increasing in $k$.
Substituting the maximum value of $k$ for this case, $r(1-p)/p-2$, into~$p_*$ and then this into~\eq{ugh} and simplifying,
we find that~\eq{smv1} is bounded above by
\ba{
\frac{r}{p}\left(\frac{r/p-2/(1-p)}{r/p-3}\right)^{r/p-r-2} \leq \frac{r}{p} e^{3-2/(1-p)}\leq \frac{  r }{p} e,
}
where the first inequality follows since $r/p-r-2\leq r/p-3$ and that for $a,x>0$
$\left(1+\tsfrac{a}{x}\right)^x\leq e^{a}$.

\end{proof}

We need the following easy corollary of Lemma~\ref{lemnbstbd}.
\begin{lemma}\label{lemnbstbd1}
If for $A\subseteq\IZ_+$, $g:=g_A$ satisfies the Stein equation~\eq{nbstsol}, then
\ba{
&\sup_{k\in\IZ_+}\abs{\dr g(k)}\leq \frac{\max\{r,1\} e+1}{r(1-p)}, \\
&\sup_{k\in\IZ_+}\abs{\Delta (\dr g(k))}\leq \min\left\{1+\frac{2}{r},\frac{2-p}{r(1-p)}\right\}.
}
\end{lemma}
\begin{proof}
For the first assertion, since $g$ solves the Stein equation~\eq{nbstsol}, 
\ba{
\abs{r(1-p)\dr g(k)} &\leq \abs{p k g(k)+\I[k\in A]-\IP(Y\in A)} \\
	&\leq \abs{p k g(k)}+\abs{\I[k\in A]-\IP(Y\in A)} \\
	&\leq \max\{r,1\} e+1,
}
where we have used Lemma~\ref{lemnbstbd}.

For the second assertion, it is easy to see that
\ba{
\Delta (\dr g(k))=\frac{r+k+1}{r}\D g(k+1)-\frac{k}{r}\D g(k),
}
and the lemma follows after taking the absolute value, applying the triangle inequality, and
judiciously using Lemma~\ref{lemnbstbd}.
\end{proof}

\begin{proof}[Proof of Theorem~\ref{thmnbst}]
Following the usual Stein's method machinery,
for $Y\sim\NB(r,p)$ and $g:=g_A$ solving~\eq{nbstsol} for $A\subseteq\IZ_+$, we have
\ban{
\dtv(\law(W)&,\NB(r,p))=\sup_{A\subseteq \IZ_+}|\IE[\I[W\in A]-\IP(Y\in A)]| \notag\\ 
&=\sup_{A\subseteq \IZ_+}|\IE [(1-p)(r+W) g_A(W+1)-Wg_A(W)]|, \notag\\
&=p\sup_{A\subseteq \IZ_+}|\IE[\mu\dr g_A(W)- Wg_A(W)]|. \notag
} 
Lemma~\ref{lemnbrt} implies that for $g:=g_A$,
\ba{
p\IE[\mu\dr g(W)&- Wg(W)]=p\mu\IE[\dr g(W)- \dr g(W^{*_r})]  \\
	&=p\mu\IE[(\dr g(W)- \dr g(W^{*_r}))\I_B] \\
	&\qquad+p\mu\IE[(\dr g(W)- \dr g(W^{*_r}))\I_{B^c}] \\
	&=:R_1+R_2.
}
Using that $\mu=r(1-p)/p$, we have
\ba{
p\mu\abs{\dr g(W)- \dr g(W^{*_r})}\leq 2r(1-p)\sup_{k\in\IZ_+}\abs{\dr g(k)}, 
}
and so Lemma~\ref{lemnbstbd1} implies that $\abs{R_2}\leq 2(e\max\{1,r\}+1)\IP(B^c)$.

To bound $\abs{R_1}$, we write
\ba{
|\dr g(W)- \dr g(W^{*_r})|&=\bigg|\I[W>W^{*_r}]\sum_{k=0}^{W-W^{*_r}-1}\D \dr g(W^{*_r}+k)\\
	&\quad-\I[W^{*_r}>W]\sum_{k=0}^{W^{*_r}-W-1}\D \dr g(W+k)\bigg|,\\
	&\leq \sup_{k\in\IZ_+}\abs{\Delta (\dr g(k))}\abs{W^{*_r}-W}.
}
Combining this with the bound of Lemma~\ref{lemnbstbd1}, we find
\ba{
\abs{R_1}\leq\min\{(r+2)(1-p),2-p\}\IE\abs{W^{*_r}-W}\I_{B},
}
which, upon adding to the bound on $\abs{R_2}$, yields the first bound in theorem.
The second bound is obtained from the first by choosing $B=\{W=W^{*_r}\}$.
\end{proof}

\section{Preferential attachment proof}\label{secpa}

In this section we prove Theorem~\ref{thmmain} following the strategy 
of proof of the main result of Section~6 of~\cite{prr10}, which is a special case of 
our results (it will likely help the reader to first understand the proof there).
We use $C_{m,\delta}$
to denote a constant only depending on $m$ and $\delta$ which may change from line to line.

Theorem~\ref{thmmain} easily follows by the triangle inequality applied to the following three claims.
If $I$ is uniform on
$\{1, \ldots, n\}$%~$U$ is uniform on $(0,1)$, 
independent of $W_{n,i}$ defined to be the in-degree of vertex $i$ in $G_n^{m,\delta}$,
and $\mu_{n,i}:=\IE W_{n,i}$, then
\begin{enumerate}
\item\label{it1} $\dtv\bklr{\law(W_{n,I}), \NB(m+\delta, \tsfrac{m+\delta}{\mu_{n,I}+m+\delta})}\leq
C_{m,\delta}\frac{\log(n)}{n}$,
\item\label{it2} $\dtv\bklr{ \NB(m+\delta, \tsfrac{m+\delta}{\mu_{n,I}+m+\delta}), \NB(m+\delta, (I/n)^{1/(2+\delta/m)})}\leq
C_{m,\delta}\frac{\log(n)}{n}$,
\item\label{it3} $\dtv\bklr{  \NB(m+\delta, (I/n)^{1/(2+\delta/m)}), \K(m,\delta) }\leq C_{m,\delta}\frac{\log(n)}{n}$.
\end{enumerate}

The proofs of Items~\ref{it2} and~\ref{it3} are relatively straightforward,
while the proof of Item~\ref{it1} uses the following result which we show using Stein's method
(i.e.\ Theorem~\ref{thmnbst}).

\begin{theorem} \label{thm8}
Retaining the notation and definitions above, we have
\ba{
       \dtv\bklr{\law(W_{n,i}), \NB(m+\delta, \tsfrac{m+\delta}{\mu_{n,i}+m+\delta})}
                \leq \frac{C_{m,\delta}}{i}.
}
\end{theorem}

The layout of the remainder of this section is as follows. We 
first collect and prove some lemmas necessary for
the proof of Items~\eq{it1}-\eq{it3} above and then prove these results.
We prove Theorem~\ref{thm8} last, since it is relatively involved. 

Since $G_n^{m,\delta}$ is constructed from $G_{nm}^{1,\delta/m}$ it will be 
helpful to denote $W_{k,j}\t{1,\eps}$ to be the in-degree of vertex $j$ 
in $G_k^{1,\eps}$ for $k\geq j-1$, where we set $W_{j-1,1}\t{1,\eps}:=0$.
The first lemma is useful for computing moment information; 
it is a small variation of a special case of the remarkable results of \cite{mor05}, see also Proposition~8.9 in~\cite{vdh12}.
\begin{theorem}\label{thmmart}\cite{mor05}
If $j\geq1$ and $\eps>-1$, then the sequence of random variables 
\ba{
\frac{\Gamma\left(k+\tsfrac{1+\eps}{2+\eps}\right)}{\Gamma\left(k+1\right)}(W_{k,j}\t{1,\eps}+1+\eps)
}
is a martingale for $k\geq j-1$, where we take $W_{j-1,j}\t{1,\eps}:=0$. In particular, for $k\geq j-1$,
\ba{
\IE W_{k,j}\t{1,\eps}+1+\eps
=(1+\eps)\frac{\Gamma\left(k+1\right)\Gamma\left(j-1+\tsfrac{1+\eps}{2+\eps}\right)}{\Gamma\left(k+\tsfrac{1+\eps}{2+\eps}\right)\Gamma\left(j\right)}.
}
\end{theorem}

We also need asymptotic estimates for the ratio of gamma functions. The next result
follows from Stirling's approximation. 
\begin{lemma}\label{lemgamrat}
For fixed $a,b>0$, as $z\to\infty$, 
\ba{
\frac{\Gamma(z+a)}{\Gamma(z+b)}=z^{a-b}+\bigo(z^{a-b-1}).
}
\end{lemma}

The next lemma provides a nice asymptotic expression
for  expectations appearing in the proofs below.

\begin{lemma}\label{lem3}
If $n\geq i$ and $-\delta< m\in \IN$,
and $\mu_{n,i}\t{m,\delta}:=\IE W_{n,i}$, then
\ba{
\left|\frac{\mu_{n,i}\t{m,\delta}}{m+\delta}+1-\left(\frac{n}{i}\right)^{1/(2+\delta/m)}\right|\leq C_{m,\delta} \left(\frac{n}{i}\right)^{1/(2+\delta/m)}\frac{1}{i},
}
%\ba{
%\left|\frac{m+\delta}{\IE W_{n,i}+m+\delta}-\left(\frac{i}{n}\right)^{1/(2+\delta/m)}\right|\leq \frac{t}{\left(\left(\frac{n}{i}\right)^{1/(2+\delta/m)}-t\right)\left(\frac{n}{i}\right)^{1/(2+\delta/m)}},
%}
\ba{
\left|\frac{m+\delta}{\mu_{n,i}\t{m,\delta}+m+\delta}-\left(\frac{i}{n}\right)^{1/(2+\delta/m)}\right|\leq \frac{C_{m,\delta}\left(\frac{i}{n}\right)^{1/(2+\delta/m)}}{i}.
}
\end{lemma}

\begin{proof}
The second inequality follows directly from the first.  For the first assertion,
Theorem~\ref{thmmart} implies that for $\eps>-1$  and $\mu_{k,j}\t{1,\eps}:=\IE W_{k,j}\t{1,\eps}$ for $k\geq j-1$,
\ba{
\mu_{k,j}\t{1,\eps}=(1+\eps)\left[\frac{\Gamma(j-1+\tsfrac{1+\eps}{2+\eps})\Gamma(k+1)}{\Gamma(j)\Gamma(k+\tsfrac{1+\eps}{2+\eps})}-1\right].
}
The construction of $G_{n}^{m,\delta}$ implies that
\ba{
\mu_{n,i}\t{m,\delta}=\sum_{j=1}^m\mu_{nm,(i-1)m+j}\t{1,\delta/m},
}
so we find that 
\ban{ 
&\mu_{n,i}\t{m,\delta}+(m+\delta) \notag \\
&=(1+\tsfrac{\delta}{m})\frac{\Gamma(nm+1)}{\Gamma(nm+\tsfrac{1+\delta/m}{2+\delta/m})}\left[\sum_{j=1}^m\frac{\Gamma((i-1)m+j-1+\tsfrac{1+\delta/m}{2+\delta/m})}{\Gamma((i-1)m+j)}\right]. \label{muas1}
}
Using now Lemma~\ref{lemgamrat} for the ratios of gamma functions, we find for $i>1$,
\ba{
\frac{\mu_{n,i}\t{m,\delta}}{m+\delta}+1&=\frac{1}{m}\left((nm)^{\frac{1}{2+\delta/m}}+\bigo(n^{-\frac{1+\delta/m}{2+\delta/m}})\right) \\
&\qquad\times\left(m((i-1)m)^{-\frac{1}{2+\delta/m}}+\sum_{j=1}^m\bigo(i^{-\frac{3+\delta/m}{2+\delta/m}})\right).
}
The lead term equals $(n/i)^{1/(2+\delta)}$  (up to the error in changing $i-1$ to $i$), and the second order term is 
easily seen to be as desired. In the case that $i=1$, similar arguments starting from~\eq{muas1} yield the appropriate
complementary result.
\end{proof}

To prove Items~\ref{it2} and ~\ref{it3} we
have to bound the total variation distance between negative binomial
distributions having different `$p$' parameters. The next
result is sufficient for our purposes.

\begin{lemma}\label{lemrr}
If $r>0$ and $0\leq \eps<p\leq 1$, then
\ban{
        \dtv \bklr{\NB(r,p),\NB(r,p-\eps)}
                \leq \frac{r \eps}{p-\eps}. \label{28}
}
\end{lemma}
\begin{proof}
Proposition~2.5 of~\cite{adjo06} implies
that for $r>0$ (their statement is for $r\in\IN$, but the
same proof works for all $r>0$),
\ban{
 \dtv \bklr{\NB(r,p),\NB(r,p-\eps)}=(r+l-1)\int_{p-\eps}^p q(u) du, \label{ag1}
}
where $0\leq q(u) \leq 1$ and 
\ba{
l\leq \frac{r (1-p+\eps)}{(p-\eps)}+1.
}
Using these bounds on $q$ and $l$ in~\eq{ag1} implies the lemma.
\end{proof}

%\begin{lemma}\label{lemrr}
%If $r\geq1$ and $0\leq \eps<p\leq 1$, then
%\ban{
%        \dtv \bklr{\NB(r,p),\NB(r,p-\eps)}
%                \leq \frac{2r\eps}{p}\leq\frac{2r \eps}{p-\eps}. \label{28}
%}
%\end{lemma}
%\begin{proof}
%The second inequality is obvious. To prove the first, let % $X\sim\NB(r,p-\eps)$ and 
%$Y\sim\NB(r,p)$. 
%Following the Stein's method framework in the proof of Theorem~\ref{thmnbst},
%\ban{
%&\dtv(\NB(r,p),\NB(r,p-\eps)) \notag\\
%&\qquad=\sup_{A\subseteq\IZ_+}\abs{\IE [ (1-p+\eps) (r+Y)g_A(Y+1)-Yg_A(Y)]}, \label{ag}
%}
%where $g_A$ satisfies \eq{nbstsol} with $p$ replaced by $p-\eps$. 
%Since $r\geq1$, Lemma~3 of~\cite{brph99} implies that 
%$\sup_{k\in\IZ_+}\abs{g_A(k)}\leq 2$
%and so by Theorem~\ref{thmbp} we know that 
%\ba{
%\IE Y g_A(Y)= (1-p)\IE (r+Y) g_A(Y+1).
%}
%Substituting this last expression into~\eq{ag} and pushing the absolute
%value into the expectation implies
%\ba{
%\dtv(\NB(r,p),\NB(r,p-\eps))\leq \sup_{k\in\IZ_+}\abs{g(k)}\IE[Y+r]\eps 
%\leq\frac{2r\eps}{p},
%}
%where we have used that $\IE Y=r(1-p)/p$ and the bound on $\abs{g(k)}$ above.
%\end{proof}

Our final lemma is useful for handling total variation distance
for conditionally defined random variables.

\begin{lemma}\label{lem4} Let $W$ and $V$ be random variables and let $X$ be a
random element defined on the same probability space. Then
\be{
        \dtv(\law(W), \law(V))\leq \IE \dtv(\law(W|X), \law(V|X)).
}
\end{lemma}
\begin{proof} If $f:\IR\rightarrow[0,1]$, then
\be{
        \abs{\IE [f(W)-f(V)]}\leq \IE\abs{\IE[f(W)-f(V)|X]}
        \leq \IE \dtv\bklr{\law(W|X), \law(V|X)}. \qedhere
}
\end{proof}

\begin{proof}[Proof of Theorem~\ref{thmmain}]
%We use the following fact which is Lemma~\ref{lemrr} in the appendix.
%If $r\geq1$ and $0\leq \eps<p\leq 1$. then
%\ben{
%        \dtv \bklr{\NB(r,p),\NB(r,p-\eps)}
%                \leq \frac{2r\eps}{p}< \frac{2r \eps}{p-\eps}. \label{28}
%}
%The second inequality of \eq{28} is obvious.
%To see the first inequality, note first that 
%if $X_1,\ldots,X_n$ are i.i.d.\ and $Y_1,\ldots,Y_n$ are i.i.d.\ then
%\ba{
%\dtv(\law(\sum_{j=1}^n X_i),\law(\sum_{j=1}^n Y_i))&\leq \IP(\cup_i\{X_i\not=Y_i\})\leq\sum_{j=1}^n \IP(X_i\not=Y_i).
%%	&\leq\sum_{i=1}^n\dtv(\law(X_i),\law(Y_i)).
%}
%Furthermore, it is straightforward to construct a coupling of 
%random variables $X$ and $Y$ having geometric distributions with respective 
%parameters $p$ and $p-\eps$ such
%that $\dtv(\law(X),\law(Y))\leq p/\eps$ (see (6.1) in the proof of Theorem~6.1 in \cite{prr10}).
%combining these two facts shows~\eq{28}.
 
Using \eq{28} and Lemma~\ref{lem3} we easily obtain
\be{
        \dtv\bklr{ \NB(m+\delta, \tsfrac{m+\delta}{\mu_{n,i}+m+\delta}), \NB(m+\delta, (i/n)^{1/(2+\delta/m)})}  \leq \frac{C_{m,\delta}}{i}
}
and applying Lemma~\ref{lem4}
we find
\be{
        \dtv\bklr{ \NB(m+\delta, \tsfrac{m+\delta}{\mu_{n,I}+m+\delta}), \NB(m+\delta, (I/n)^{1/(2+\delta/m)})}
                \leq \frac{C_{m,\delta}\log(n)}n,
}
which is Item~\ref{it2} above. Now, we couple $U$ to $I$ by writing $U=I/n-V$,
where $V$ is uniform on $(0,1/n)$
and independent of $I$.
From here, use \eq{28}, Lemma~\ref{lem4}, and then the easy fact that
for $i\geq1$ and $0<a<1$,
\ba{
i^a-(i-1)^a\leq i^{a-1},
}
to find
\ba{
&\dtv\bklr{ \NB(m+\delta, (I/n)^{1/(2+\delta/m)}), \K(m,\delta) } \\
&\qquad=\dtv\bklr{  \NB(m+\delta, (I/n)^{1/(2+\delta/m)}), \NB(m+\delta, U^{1/(2+\delta/m)}) } \\
        	&\qquad\leq \frac{C_{m,\delta}}{n}\sum_{i=1}^n \frac{(i/n)^{1/(2+\delta/m)}-((i-1)/n)^{1/(2+\delta/m)}}{(i/n)^{1/(2+\delta/m)}} \\
               & \qquad\leq \frac{C_{m,\delta}\log(n)}n,
}
which is Item~\ref{it3} above. Finally, applying Lemma~\ref{lem4} to
Theorem~\ref{thm8} yields the claim
in Item~\ref{it1} above
so that Theorem~\ref{thmmain} is proved.
\end{proof}

The remainder of the section is devoted to the proof of Theorem~\ref{thm8}.
Since we want to apply our negative binomial 
approximation framework we must first construct a random variable
having the $(m+\delta)$-equilibrium distribution of $W_{n,i}:=W_{n,i}\t{m,\delta}$. According to Definition~\ref{req}, we
first construct a variable having the size bias
distribution of $W_{n,i}$.  To facilitate this construction we need some
auxiliary variables.

We mostly work with 
$G_n^{m,\delta}$ through the intermediate construction of
$G_{nm}^{1,\delta/m}$ discussed in the introduction. 
To fix notation, if for $k\geq j$, $W_{k,j}\t{1,\delta/m}$ is the degree
of vertex $j$ in $G_{k}^{1,\delta/m}$, then we write
\ban{
W_{n,i}=\sum_{j=1}^m W_{nm,m(i-1)+j}\t{1,\delta/m}. \label{decomp1}
}
Further, if we let $X_{j,i}\t{\delta/m}$ be the indicator that
vertex $j$ attaches to vertex $i$ in $G_j^{1,\delta/m}$ (and hence also in
$G_k^{1,\delta/m}$ for $j\leq k \leq mn$), then we also have
\ban{
W_{nm,m(i-1)+j}\t{1,\delta/m}=\sum_{k=m(i-1)+j}^{mn}X_{k,m(i-1)+j}\t{\delta/m}. \label{decomp2}
}

The following well-known result allows us to use the decomposition of $W_{n,i}$ into a sum of indicators
as per~\eq{decomp1} and~\eq{decomp2} to
size bias $W_{n,i}$; see e.g.\ Proposition~2.2 of~\cite{cgs11} and the
discussion thereafter.

\begin{proposition}\label{prop2}
Let $X_1, \ldots, X_n$ be zero-one random variables such that $\IP(X_j=1)=p_j$.
For each $k=1,\ldots, n$, let $(X_j^{(k)})_{j\not=k}$ have the distribution of
$(X_j)_{j\not=k}$ conditional on $X_k=1$.
If $X=\sum_{j=1}^n X_j$, $\mu=\IE[X]$,
and $K$ is chosen independent of the variables above with $\IP(K=k)=p_k/\mu$,
then $X^s=\sum_{j\not=K} X_j^{(K)} +1$ has the size bias distribution of $X$.
\end{proposition}
Roughly, Proposition~\ref{prop2} implies that in order to size bias $W_{n,i}$,
we choose an indicator
$X_{K,L}\t{\delta/m}$ where for $l=m(i-1)+1,\ldots, mi$, $k=l,\ldots,mn$, $\IP(K=k,L=l)$ is proportional to
$\IP(X_{k,l}\t{\delta/m}=1)$ (and zero for other values),
then attach vertex
$K$ to vertex $L$ and sample the remaining edges conditional on this event.
Note that given $(K,L)=(k,l)$, in the graphs $G_j^{1,\delta/m}$, $1\leq j < l$ and
$k<j\leq nm$, this conditioning does not change the original rule for generating
the preferential attachment graph given $G_{j-1}^{1,\delta/m}$.
The following lemma implies the remarkable fact that in order to generate the
graphs $G_j^{1,\delta/m}$ for $l \leq j < k$
conditional on $X_{k,l}\t{\delta/m}=1$  and $G_{l-1}$,
we attach edges following the same rule as preferential attachment, but include
the edge from vertex $k$
to vertex $l$ in the degree count.

\begin{lemma}\label{lemcondlat}
Retaining the notation and definitions above,
for $l,s \leq j<k$ we have
\ben{\label{tt}
        \IP(X_{j,s}\t{\delta/m}=1 | X_{k,l}\t{\delta/m}=1,G_{j-1}^{1,\delta/m}) = \frac{\I[s=l]+W_{j-1,s}\t{1,\delta/m}+\delta/m+1}{j(2+\delta/m)},
}
where we define $W_{j-1,j}\t{1,\delta/m}=0$.
\end{lemma}
\begin{proof}
By the definition of conditional probability, we write
\ben{\label{23}\begin{split}
&\IP(X_{j,s}\t{\delta/m}=1 | X_{k,l}\t{\delta/m}=1,G_{j-1}^{1,\delta/m})\\
         &\quad= \frac{\IP(X_{j,s}\t{\delta/m}=1| G_{j-1}^{1,\delta/m}) \IP(X_{k,l}\t{\delta/m}=1 | X_{j,s}\t{\delta/m}=1,
G_{j-1}^{1,\delta/m})}{\IP(X_{k,l}\t{\delta/m}=1| G_{j-1}^{1,\delta/m})},
\end{split}}
and we calculate the three probabilities appearing above.  
First note
\be{
        \IP(X_{j,s}\t{\delta/m}=1| G_{j-1}^{1,\delta/m}) = \frac{W_{j-1,s}\t{1,\delta/m}+1+\tsfrac{\delta}{m}}{(2+\delta/m)j-1},
}
which implies
\be{
        \IP(X_{k,l}\t{\delta/m}=1| G_{j-1}^{1,\delta/m}) = \frac{\IE [W_{k-1,l}\t{1,\delta/m}+1+\tsfrac{\delta}{m} | G_{j-1}^{1,\delta/m}]}{(2+\delta/m)k-1}
}
and
\ba{
	&\IP(X_{k,l}\t{\delta/m}=1 | X_{j,s}\t{\delta/m}=1, G_{j-1}^{1,\delta/m}) \\
               &\qquad= \frac{\IE [W_{k-1,l}\t{1,\delta/m} +1+\tsfrac{\delta}{m}| X_{j,s}\t{\delta/m}=1, G_{j-1}^{1,\delta/m}]}{(2+\delta/m)k-1}.
}
Using Theorem~\ref{thmmart}, it easy to see that
\ba{
\IE [W_{k-1,l}\t{1,\delta/m}+1+\tsfrac{\delta}{m} | G_{j-1}^{1,\delta/m}]
=\frac{\Gamma(k)\Gamma(j-1+\tsfrac{1+\delta/m}{2+\delta/m})}{\Gamma(k-1+\tsfrac{1+\delta/m}{2+\delta/m})\Gamma(j)}(W_{j-1,l}\t{1,\delta/m}+1+\tsfrac{\delta}{m}),
}
and also
\ba{
\IE [W_{k-1,l}\t{1,\delta/m}& +1+\tsfrac{\delta}{m}| X_{j,s}\t{\delta/m}=1, G_{j-1}^{1,\delta/m}] \\
&=\frac{\Gamma(k)\Gamma(j+\tsfrac{1+\delta/m}{2+\delta/m})}{\Gamma(k-1+\tsfrac{1+\delta/m}{2+\delta/m})\Gamma(j+1)}(W_{j-1,l}\t{1,\delta/m}+\I[s=l]+1+\tsfrac{\delta}{m}).
}
Combining these calculations with~\eq{23} and simplifying (using in particular that $\Gamma(x+1)=x\Gamma(x)$) implies
\ban{
&\IP(X_{j,s}\t{\delta/m}=1 | X_{k,l}\t{\delta/m}=1,G_{j-1}^{1,\delta/m}) \notag\\
         &\quad=\frac{1}{j(2+\delta/m)}\frac{(W_{j-1,l}\t{1,\delta/m}+\I[s=l]+1+\tsfrac{\delta}{m})(W_{j-1,s}\t{1,\delta/m}+1+\tsfrac{\delta}{m})}{W_{j-1,l}\t{1,\delta/m}+1+\tsfrac{\delta}{m}}. \label{tt1}
}
Considering the cases $s=l$ and $s\not=l$ separately yields that~\eq{tt1} equals~\eq{tt}.
\end{proof}

The previous lemma suggests the following (embellished) construction of
$(W_{n,i} | X_{k,l}\t{\delta/m}=1)$.  Here and below we 
denote quantities related to this construction by amending $(k,l)$. First we
generate ${G}_{l-1}^{1,\delta/m}(k,l)$, a graph with $l-1$ vertices, according to the usual
preferential attachment model. At this point, if $l\not=k$, vertex $l$ and $k$
are added to the graph, along with a vertex labeled $i'$ with an edge to it
emanating from vertex $k$. Given  ${G}_{l-1}^{1,\delta/m}(k,l)$ and these additional vertices and
edges, we generate ${G}_l^{1,\delta/m}(k,l)$ by connecting vertex $l$ to a vertex randomly
chosen from the vertices $1, \ldots, l, i'$ proportional to their ``degree weight," where
vertex $l$ has degree weight $1+\delta/m$ (from the out-edge) and $i'$ has degree one (from the in-edge
emanating from vertex $k$), and the remaining vertices have degree weight equal to their
degree plus $\delta/m$.   For $l < j <k$, we generate the graphs $G_{j}^{1,\delta/m}(k,l)$
recursively from $G_{j-1}^{1,\delta/m}(k,l)$ by connecting vertex $j$ to a vertex randomly
chosen from the vertices $1, \ldots, j, i'$ proportional to their degree weight, where
$j$ has degree weight $1+\delta/m$ (from the out-edge).  Note that none of the vertices $1,
\ldots, k-1$ connect to vertex $k$. Also define $G_k^{1,\delta/m}(k,l)=G_{k-1}^{1,\delta/m}(k,l)$.
If $l=k$, we attach vertex $k$ to $i'$ and denote the
resulting graph by ${G}_k^{1,\delta/m}(k,l)$.
For all values $(k,l)$, if
$j=k+1, \ldots, nm$, we generate $G_j^{1,\delta/m}(k,l)$ from $G_{j-1}^{1,\delta/m}(k,l)$ according to 
usual preferential
attachment among the vertices $1, \ldots, j, i'$.

We have a final bit of notation before stating relevant properties of these objects.
Denote the degree of vertex $j$ in this construction by $W_{nm,j}\t{1,\delta/m}(k,l)$ and
let also 
\ba{
W_{n,i}(k,l):=\sum_{j=1}^m W_{nm,m(i-1)+j}\t{1,\delta/m}(k,l).
}
Let $B_{k,l}$ be the event that in this construction all edges emanating from the vertices $m(i-1)+1,\dots,mi$
attach to one of the vertices $1,\ldots,m(i-1)$. In symbols,
\be{
B_{k,l}=\bigg\{
\begin{array}{l}
X_{s,j}\t{\delta/m}(k,l)=0 \text{ for all }  j\in\{m(i-1)+1,\dots,mi,i'\},\\
s\in\{m(i-1)+1,\dots,mi,i'\}/\{k\}
\end{array}
\bigg\}.
}
Finally, let $W'$ have the $r$-equilibrium distribution of $W_{n,i}$, independent of all else and define
\ban{
W_{n,i}^{*_r}=W_{n,i}(K,L)\I_{B_{K,L}}+W' \I_{B_{K,L}^c}. \label{ret}
}

\begin{lemma} \label{lem6}
Let $l\in \{m(i-1)+1,\ldots, mi\}$, $k \in \{l,\ldots,mn\}$ and 
retain the notation and definitions above.
\begin{enumerate}
\item $\law (W_{n,i}(k,l)+W_{nm,i'}\t{1,\delta/m}(k,l)) = \law  (W_{n,i} | X_{k,l}\t{\delta/m}=1).$
\item  If $(K,L)$ is a random vector such that
\be{
        \IP (K=k',L=l') = \frac{\IE X_{k',l'}\t{\delta/m}}{\IE W_{n,i}}, \quad k'\geq l'\in\{m(i-1)+1,\ldots, mi\},
}
then $W_{n,i}(K,L)+W_{nm,i'}\t{1,\delta/m}(K,L)$ has the size bias distribution of $W_{n,i}$.
\item Conditional on the event 
\ba{
\{W_{n,i}(K,L)+W_{nm,i'}\t{1,\delta/m}(K,L)=t\},
}
$\law(W_{n,i}(K,L)\I[B_{K,L}])=\law(U_{m+\delta,t}\I[B_{K,L}])$, where $U_{r,t}$ has 
the P\'olya urn distribution of Definition~\ref{req} and is independent of all else.
\item $W_{n,i}^{*_r}$ has the $(m+\delta)$-equilibrium distribution of $W_{n,i}$.
\end{enumerate}
\end{lemma}
\begin{proof}
Items~1 and~2 follow from Proposition~\ref{prop2} and Lemma~\ref{lemcondlat}. 
Item~3 follows since under the conditioning, if $\I[B_{K,L}]=1$, then $W_{n,i}(K,L)$ is 
distributed as the number of white balls \emph{drawn} in $t-1$ draws from a P\'olya urn
started with $m+\delta$ white balls and $1$ black ball (it's $t-1$
draws, rather than $t$, since 
the initial ``black ball" degree from vertex $i'$ is included in the degree count 
$W_{n,i}(K,L)+W_{nm,i'}\t{1,\delta/m}(K,L)$).
Item~4 follows from Items~1-3, using Definition~\ref{req}.
\end{proof}

\begin{proof}[Proof of Theorem~\ref{thm8}]
We apply Theorem~\ref{thmnbst} to $\law(W_{n,i})$
with $W_{n,i}^{*_r}$ as defined by~\eq{ret}. Before constructing the coupling of $\law(W_{n,i})$ required  
in Theorem~\ref{thmnbst},
we reduce the bound $\IP(W_{n,i}^{*_r}\not=W_{n,i})$. 

First note that due to the form $W_{n,i}^{*_r}$, 
we have (no matter how $\law(W_{n,i})$ is coupled)
\ban{
\IP(W_{n,i}^{*_r}\not=W_{n,i})&=\IP(W_{n,i}(K,L)\not=W_{n,i},B_{K,L})+\IP(W'\not=W_{n,i},B_{K,L}^c)\notag \\
	&\leq \IP(W_{n,i}(K,L)\not=W_{n,i})+\IP(B_{K,L}^c). \label{fbs1}
}
We bound the second term of \eq{fbs1} as follows.
For $l\in\{m(i-1)+1,\ldots,mi\}$ and $k>mi$, we directly compute
\ban{
\IP(B_{k,l})&= \prod_{j=m(i-1)+1}^{l-1} \frac{m(i-1)(1+\tsfrac{\delta}{m})+j-1}{j(2+\tsfrac{\delta}{m})-1}\prod_{j=l}^{mi}\frac{m(i-1)(1+\tsfrac{\delta}{m})+j-1}{j(2+\tsfrac{\delta}{m})} \notag\\
	&\geq \prod_{j=m(i-1)+1}^{mi}\frac{m(i-1)(1+\tsfrac{\delta}{m})+j-1}{j(2+\tsfrac{\delta}{m})} \notag\\
	&=\frac{1}{(2+\tsfrac{\delta}{m})^m}\frac{\Gamma(m(i-1)(2+\tsfrac{\delta}{m})+m)\Gamma(m(i-1)+1)}{\Gamma(m(i-1)(2+\tsfrac{\delta}{m}))\Gamma(mi+1)} \label{bg1} \\
	&=1+\bigo(1/i),\notag
}
where in the last equality use Lemma~\ref{lemgamrat}. 
If $k\in\{m(i-1)+1,\ldots,mi\}$, then 
\ba{
\IP(B_{k,l})
&= \prod_{j=m(i-1)+1}^{l-1} \frac{m(i-1)(1+\tsfrac{\delta}{m})+j-1}{j(2+\tsfrac{\delta}{m})-1}\mathop{\prod_{j=l}^{mi}}_{j\not=k}\frac{m(i-1)(1+\tsfrac{\delta}{m})+j-1}{j(2+\tsfrac{\delta}{m})},
}
which is greater than or equal to~\eq{bg1} (since the omitted term is a probability), so in either case we find
\ba{
\IP(B_{K,L}^c)=\bigo(1/i).
}
%so now we must bound
%\ba{
%\IP(K\leq mi)=\frac{1}{\IE W_{n,i}}\sum_{l=m(i-1)+1}^{mi}\sum_{k=l}^{mi}\IE X_{k,l}\t{\delta/m}.
%}
%Since 
%\ba{
%\IE X_{k,l}\t{\delta/m}=\frac{\IE W_{k-1,l}\t{1,\delta/m}+1+\delta/m}{(2+\delta)k-1},
%}
%Theorem~\ref{thmmart} implies that
%\ba{
%\IE X_{k,l}\t{\delta/m}=\left(\frac{1+\eps}{{(2+\delta)k-1}}\right)\frac{\Gamma\left(k\right)\Gamma\left(l-1+\tsfrac{1+\eps}{2+\eps}\right)}{\Gamma\left(k-1+\tsfrac{1+\eps}{2+\eps}\right)\Gamma\left(l\right)},
%}

We have only left to bound the first term of~\eq{fbs1}, for which we must first 
define the coupling of $\law(W_{n,i})$ to $W_{n,i}(K,L)$.
For each $(k,l)$ in the support of $(K,L)$, we construct
\ban{ \label{coup}
\left\{\left(X_{s,j}\t{\delta/m}(k,l),\widetilde{X}_{s,j}\t{\delta/m}\right) : mn\geq s\geq j\in\{m(i-1)+1,\ldots,mi\right\},
}
to have the distribution of the indicators of the events
vertex $s$ connects to vertex $j$ in $G_{nm}^{1,\delta/m}(k,l)$ and
$G_{nm}^{1,\delta/m}$, respectively. With this fact established,
denote
\be{
W_{nm,j}\t{1,\delta/m}(k,l)=\sum_{s=j}^{nm} X_{s,j}\t{\delta/m}(k,l)
  \quad\text{and}\quad
        \widetilde{W}_{nm}\t{1,\delta/m}=\sum_{s=j}^{nm} \widetilde{X}\t{\delta/m}_{s,j},
}
which have the distribution of vertex $j$ in the indicated graphs,
and then we set
\be{
W_{n,i}(k,l)=\sum_{j=m(i-1)+1}^{mi} W_{nm,j}\t{1,\delta/m}(k,l)
\quad\text{and}\quad
        {W}_{n,i}=\sum_{j=m(i-1)+1}^{mi} \widetilde{W}_{nm,j}\t{1,\delta/m}.
} 
From this point we bound the first term of~\eq{fbs1} via
\ban{
\IP(W_{n,i}(k,l)\not=W_{n,i})&\leq\IP\left(\mathop{\bigcup}_{j=m(i-1)+1}^{mi} \left\{ W_{nm,j}\t{1,\delta/m}(k,l)\not=\widetilde{W}_{nm,j}\t{1,\delta/m}\right\}\right)\notag\\
&\leq \sum_{j=m(i-1)+1}^{mi} \IP(W_{nm,j}\t{1,\delta/m}(k,l)\not=\widetilde{W}_{nm,j}\t{1,\delta/m}), \label{coupbd}
}
and we show each term in the sum is $\bigo(1/i)$ (still depending on $m,\delta$, but not on $k,l$),
which establishes the theorem.

The constructions for different orders of $j,k,l$ are slightly different, so 
assume that $j < l< k$.  
Let $U_{s,j}(k,l)$ be independent uniform $(0,1)$ random variables
and for the sake of brevity, let $w=1+\delta/m$. First define 
\be{
        X_{j,j}\t{\delta/m}(k,l)=\I\left[U_{s,j}(k,l)<\frac{w}{j(2+\delta/m)-1)}\right]
%        \quad\text{and}\quad
%        \widetilde{X}_{s,j}\t{\delta/m}=X_{s,j}\t{\delta/m}(k,l).
}
and for $j<s<l$, given $W_{s-1,j}\t{1,\delta/m}(k,l)$,
\be{
  X_{s,j}\t{\delta/m}(k,l)=\I\left[U_{s,j}(k,l)<\frac{W_{s-1,j}\t{1,\delta/m}(k,l)+w}{s(2+\delta/m)-1)}\right].
}
Also let $\widetilde{X}_{s,j}\t{\delta/m}=X_{s,j}\t{\delta/m}(k,l)$ for $j\leq s <l$. That is, 
we can perfectly couple the degrees of vertex $j$ in the two graphs up until vertex $l$ arrives.
Now, for $l\leq s<k$, given $W_{s-1,j}\t{1,\delta/m}(k,l)$
and $\widetilde{W}_{s-1,j}\t{1,\delta/m}$ 
%(and notice these two variables are actually equal for $s=l$),
define
\ban{   
  X_{s,j}\t{\delta/m}(k,&l)=\I\bbbkle{U_{s,j}(k,l)<\frac{W_{s-1,j}\t{1,\delta/m}(k,l)+w}{s(2+\delta/m)}}, \label{30} \\
      &\widetilde{X}_{s,j}\t{\delta/m}=\I\bbbkle{U_{s,j}(k,l)<
            \frac{\widetilde{W}_{s-1,j}\t{1,\delta/m}+w}{s(2+\delta/m)-1}}. \label{31}
}
Set $X_{k,j}\t{\delta/m}(k,l)=0$ and $\widetilde{X}_{s,j}\t{\delta/m}$ as in~\eq{31} with $s=k$
and for $s>k$, define 
\ba{
  X_{s,j}\t{\delta/m}(k,&l)=\I\bbbkle{U_{s,j}(k,l)<\frac{W_{s-1,j}\t{1,\delta/m}(k,l)+w}{s(2+\delta/m)-1}}, \notag \\
      &\widetilde{X}_{s,j}\t{\delta/m}=\I\bbbkle{U_{s,j}(k,l)<
            \frac{\widetilde{W}_{s-1,j}\t{1,\delta/m}+w}{s(2+\delta/m)-1}}. 
}

For $j<l<k$, we have jointly and recursively defined the variables
$X_{s,j}\t{\delta/m}(k,l)$ and $\widetilde{X}_{s,j}\t{\delta/m}$, and it is clear they are
distributed as claimed above with $W_{nm,j}\t{1,\delta/m}(k,l)$ and
$\widetilde{W}_{nm,j}\t{1,\delta/m}$ the required degree counts.
Note also 
$\widetilde{X}_{s,j}\t{\delta/m}\geq X_{s,j}\t{\delta/m}(k,l)$
and $\widetilde{W}_{s,j}\t{1,\delta/m}\geq W_{s,j}\t{1,\delta/m}(k,l)$
and now define the event
\ba{
A_{s,j}(k,l)=\left\{\min\left\{j\leq t\leq nm : \widetilde{X}_{t,j}\t{\delta/m}\not= X_{t,j}\t{\delta/m}(k,l) \right\}=s\right\}.
}
Using that $\widetilde{W}_{s-1,j}\t{1,\delta/m}=W_{s-1,j}\t{1,\delta/m}(k,l)$ under
$A_{s,j}(k,l)$, we have
\ba{                              
        &\IP\left(W_{nm,j}\t{1,\delta/m}(k,l)\not=\widetilde{W}_{nm,j}\t{1,\delta/m}\right)
                = \IP\left(\bigcup_{s=j}^{nm} A_{s,j}(k,l)\right) \\
                &\leq  \sum_{s=l}^{k}
                        \IP\left( A_{s,j}(k,l) \bigcap\left\{\frac{W_{s-1,j}\t{1,\delta/m}(k,l)+w}{s(2+\tsfrac{\delta}{m})}< U_{s,j}(k,l) < \frac{\widetilde{W}_{s-1,j}\t{1,\delta/m}+w}{s(2+\tsfrac{\delta}{m})-1}\right\}\right) \\
                        &\leq \IE \widetilde{X}_{k,j}\t{\delta/m}+ \sum_{s=l}^{k-1}
                        \IP\left( \frac{W_{s-1,j}\t{1,\delta/m}+w}{s(2+\tsfrac{\delta}{m}}< U_{s,j}(k,l) < \frac{\widetilde{W}_{s-1,j}\t{1,\delta/m}+w}{s(2+\tsfrac{\delta}{m})-1}\right). 
}
Now using Theorem~\ref{thmmart}, the estimates in Lemma~\ref{lem3}, and the fact that
$j,l\in\{m(i-1)+1,\ldots,mi\}$, we find
\ba{
\IP&\left(W_{nm,j}\t{1,\delta/m}(k,l)\not=\widetilde{W}_{nm,j}\t{1,\delta/m}\right) \\
	&\leq \frac{\IE W_{k-1,j}\t{1,\delta/m}+w}{k(2+\delta)-1} +\sum_{s=l}^{k-1} \left(\IE W_{s-1,j}\t{1,\delta/m}+w\right)\left( \frac{1}{s(2+\tsfrac{\delta}{m})-1}-\frac{1}{s(2+\tsfrac{\delta}{m})} \right) \\
         &\leq C_{m,\delta}\Big[\left(\frac{k}{j}\right)^{1/(2+\delta/m)}\frac{1}{k}+\left(\frac{k}{j}\right)^{1/(2+\delta/m)}\frac{1}{jk} \\
                  &\qquad+\sum_{s= l}^{\infty} \left(\left(\frac{s}{j}\right)^{1/(2+\delta/m)}\frac{1}{s^{2}}+ \left(\frac{s}{j}\right)^{1/(2+\delta/m)}\frac{1} {j s^{2}}\right)\Big]\leq C_{m,\delta}/i.
}
For the case $l<j<k$, the coupling is similar to that above, except it starts
from~\eq{30} and~\eq{31} for $j\leq s <k$; the probability estimates are also similar.
If $j>k$, then it is easy to see that the variables can be perfectly coupled.
If $j=k$ or $j<l=k$, then the analog of the coupling above can only differ if the edge 
emanating from vertex $k$ connects to $j$ in $G_{k}^{1,\delta/m}$, which occurs with chance of order
\ba{
\left(\frac{k}{j}\right)^{1/(2+\delta/m)}\frac{1}{k}=\bigo(1/i).
}  
Thus, for any $k,l$ in the support of $(K,L)$ and $j\in\{m(i-1)+1,\ldots,mi\}$,
each of the $m$ terms in the sum~\eq{coupbd} is bounded above by $C_{m,\delta}/i$,
which establishes the result.
\end{proof}

\section*{Acknowledgments}
The author thanks Erol Pek\"oz and Adrian R\"ollin for
their enthusiastic support for this project.

\def\polhk#1{\setbox0=\hbox{#1}{\ooalign{\hidewidth
  \lower1.5ex\hbox{`}\hidewidth\crcr\unhbox0}}}


\begin{thebibliography}{10}

\bibitem{adjo06}
J.~A. Adell and P.~Jodr{\'a}.
\newblock Exact {K}olmogorov and total variation distances between some
  familiar discrete distributions.
\newblock {\em J. Inequal. Appl.}, pages Art. ID 64307, 8, 2006.

\bibitem{baal99}
A.-L. Barab{\'a}si and R.~Albert.
\newblock Emergence of scaling in random networks.
\newblock {\em Science}, 286(5439):509--512, 1999.

\bibitem{bhj92}
A.~D. Barbour, L.~Holst, and S.~Janson.
\newblock {\em Poisson approximation}, volume~2 of {\em Oxford Studies in
  Probability}.
\newblock The Clarendon Press Oxford University Press, New York, 1992.
\newblock Oxford Science Publications.

\bibitem{brst01}
B.~Bollob{\'a}s, O.~Riordan, J.~Spencer, and G.~Tusn{\'a}dy.
\newblock The degree sequence of a scale-free random graph process.
\newblock {\em Random Structures Algorithms}, 18(3):279--290, 2001.

\bibitem{brph99}
T.~C. Brown and M.~J. Phillips.
\newblock Negative binomial approximation with {S}tein's method.
\newblock {\em Methodol. Comput. Appl. Probab.}, 1(4):407--421, 1999.

\bibitem{brxi01}
T.~C. Brown and A.~Xia.
\newblock Stein's method and birth-death processes.
\newblock {\em Ann. Probab.}, 29(3):1373--1403, 2001.

\bibitem{cgs11}
L.~H.~Y. Chen, L.~Goldstein, and Q.-M. Shao.
\newblock {\em Normal approximation by {S}tein's method}.
\newblock Probability and its Applications (New York). Springer, Heidelberg,
  2011.

\bibitem{gore97}
L.~Goldstein and G.~Reinert.
\newblock Stein's method and the zero bias transformation with application to
  simple random sampling.
\newblock {\em Ann. Appl. Probab.}, 7(4):935--952, 1997.

\bibitem{joko77}
N.~L. Johnson and S.~Kotz.
\newblock {\em Urn models and their application}.
\newblock John Wiley \& Sons, New York-London-Sydney, 1977.
\newblock An approach to modern discrete probability theory, Wiley Series in
  Probability and Mathematical Statistics.

\bibitem{kla00}
B.~Klar.
\newblock Bounds on tail probabilities of discrete distributions.
\newblock {\em Probab. Engrg. Inform. Sci.}, 14(2):161--171, 2000.

\bibitem{mor05}
T.~F. M{\'o}ri.
\newblock The maximum degree of the {B}arab\'asi-{A}lbert random tree.
\newblock {\em Combin. Probab. Comput.}, 14(3):339--348, 2005.

\bibitem{pek96}
E.~Pek{\"o}z.
\newblock Stein's method for geometric approximation.
\newblock {\em J. Appl. Probab.}, 33(3):707--713, 1996.

\bibitem{rope10}
E.~Pek{\"o}z and A.~R{\"o}llin.
\newblock New rates for exponential approximation and the theorems of {R}\'enyi
  and {Y}aglom.
\newblock {\em Ann. Probab.}, 39(2):587--608, 2011.

\bibitem{prr10}
E.~Pek{\"o}z, A.~R{\"o}llin, and N.~Ross.
\newblock Total variation error bounds for geometric approximation.
\newblock \url{http://arxiv.org/abs/1005.2774}, 2010.
\newblock To appear in \emph{Bernoulli}.

\bibitem{prr11}
E.~Pek{\"o}z, A.~R{\"o}llin, and N.~Ross.
\newblock Degree asymptotics with rates for preferential attachment random
  graphs.
\newblock \url{http://arxiv.org/abs/1108.5236}, 2011.
\newblock To appear in \emph{Ann. Appl. Probab.}

\bibitem{piro12}
J.~Pitman and N.~Ross.
\newblock {A}rchimedes, {G}auss, and {S}tein.
\newblock \url{http://arxiv.org/pdf/1201.4422.pdf}, 2012.
\newblock To appear in \emph{Notices of the AMS}.

\bibitem{ros11}
N.~Ross.
\newblock Fundamentals of {S}tein's method.
\newblock {\em Probability Surveys}, 8:210--293, 2011.
\newblock \url{http://dx.doi.org/10.1214/11-PS182}.

\bibitem{rtv07}
A.~Rudas, B.~T{\'o}th, and B.~Valk{\'o}.
\newblock Random trees and general branching processes.
\newblock {\em Random Structures Algorithms}, 31(2):186--202, 2007.

\bibitem{vdh12}
R.~Van Der~Hofstad.
\newblock Random graphs and complex networks.
\newblock \url{http://www.win.tue.nl/~rhofstad/NotesRGCN.pdf}, 2012.
\newblock Version of 30 {J}anuary 2012.

\end{thebibliography}
\end{document}